\documentclass[a4paper]{amsart}

\usepackage{mathrsfs}
\usepackage{amsfonts}
\usepackage{amssymb}
\usepackage{dsfont}
\usepackage[bookmarks=true]{hyperref}
\hypersetup{
pdfpagemode=UseNone,
pdfstartview=FitH,
pdfdisplaydoctitle=true,
pdfborder={0 0 0}, 
pdftitle={Variational estimates for averages and truncated singular integrals along the prime numbers},
pdfauthor={Mariusz Mirek, Bartosz Trojan, and Pavel Zorin-Kranich},
pdflang=en-GB
}

\usepackage[english,polish]{babel}

\newcommand{\norm}[1]{{\left\lvert #1 \right\rvert}}
\newcommand{\abs}[1]{{\lvert {#1} \rvert}}
\newcommand{\sabs}[1]{{\left\lvert {#1} \right\rvert}}

\newcommand{\vnorm}[1]{{\left\lVert #1 \right\rVert}}

\renewcommand{\atop}[2]{\substack{{#1}\\{#2}}}

\newcommand{\ZZ}{\mathbb{Z}}
\newcommand{\NN}{\mathbb{N}}
\newcommand{\RR}{\mathbb{R}}
\newcommand{\PP}{\mathbb{P}}
\newcommand{\TT}{\mathbb{T}}

\newcommand{\calV}{\mathcal{V}}
\newcommand{\calF}{\mathcal{F}}
\newcommand{\calO}{\mathcal{O}}

\newcommand{\pl}[1]{\foreignlanguage{polish}{#1}}

\newtheorem{theorem}{Theorem}
\newtheorem{proposition}{Proposition}[section]
\newtheorem{lemma}{Lemma}

\newcounter{thm}

\newtheorem{main_theorem}[thm]{Theorem}

\newcommand{\ind}[1]{{\mathds{1}_{{#1}}}}

\title[Variational estimates]
{Variational estimates for averages and truncated \\
 singular integrals along the prime numbers}


\author{Mariusz Mirek}
\address{Mariusz Mirek \\
	Universit\"{a}t Bonn \\
	Mathematical Institute\\
	Endenicher Allee 60\\
	D--53115 Bonn \\
	Germany \&
	\pl{Instytut Matematyczny\\
	Uniwersytet Wroc{\lll}awski\\
	Pl. Grun\-waldzki 2/4\\
	50-384 Wroc{\lll}aw}\\
	Poland}
 \email{mirek@math.uni-bonn.de}

\author{Bartosz Trojan}
\address{
	Bartosz Trojan\\
	\pl{Instytut Matematyczny\\
	Uniwersytet Wroc{\lll}awski\\
	Pl. Grun\-waldzki 2/4\\
	50-384 Wroc{\lll}aw}\\
	Poland}
\email{trojan@math.uni.wroc.pl}

\author{Pavel Zorin-Kranich}
\address{Pavel Zorin-Kranich \\
	Universit\"{a}t Bonn \\
	Mathematical Institute\\
	Endenicher Allee 60\\
	D--53115 Bonn \\
	Germany}
\email{pzorin@math.uni-bonn.de}

\thanks{
Mariusz Mirek and Bartosz Trojan were partially supported by NCN grant DEC--2012/05/D/ST1/00053.\\
Pavel Zorin-Kranich was partially supported by the ISF grant 1409/11.}

\begin{document}
\selectlanguage{english}

\begin{abstract}
  We prove, in a unified way, $r$-variational estimates, $r>2$, on
  $\ell^{s}(\ZZ)$ spaces, $s \in (1, \infty)$, for averages and truncated
  singular integrals along the set of prime numbers. Moreover,
  we obtain an improved growth rate of the bounds as $r\to 2$.
\end{abstract}
\maketitle

\section{Introduction}
Let $(X, \mathcal{B}, \mu)$ be a $\sigma$-finite measure space endowed
with an invertible measure preserving transformation $T$. In
this article we obtain, for every $s \in (1, \infty)$ and $f\in L^s(X,
\mu)$, variational bounds for the ergodic averages along the set of
prime numbers $\PP$
\begin{equation}
  \label{eq:a1}
  \mathcal A_N f(x) = \frac{1}{N} \sum_{p\in\PP_N}  f\big(T^px\big)\log p,
\end{equation}
and the truncated  Hilbert transforms
\begin{equation}
  \label{eq:s1}
  \mathcal H_Nf(x) = \sum_{p\in \pm\PP_N}f\big(T^px\big)\frac{\log|p|}{p}
\end{equation}
where $\PP_N=\PP\cap[1, N]$.

The study of pointwise convergence for averaging operators with arithmetic features was initiated
by Bourgain in \cite{bou1, bou2} and \cite{bou}, where pointwise convergence along polynomials was
proven. In \cite{bou-p,bou} Bourgain (see also Wierdl \cite{wrl}) proved that for every
$s \in (1, \infty]$ there is a constant $C_s>0$ such that
\begin{align}
	\label{eq:1}
	\big\|\sup_{N\in\NN}|\mathcal{A}_Nf|\big\|_{L^s}\le C_s\|f\|_{L^s}.
\end{align}
Recently, in \cite{mt2} it has been shown that for every $s \in (1, \infty)$ there is a constant
$C_s>0$ such that
\begin{align}
	\label{equ:2}
	\big\|\sup_{N\in\NN}|\mathcal{H}_Nf|\big\|_{L^s}\le C_s\|f\|_{L^s}.
\end{align}
These maximal inequalities combined with some oscillation estimates were used to prove the
pointwise convergence of $\mathcal A_Nf$ and $\mathcal H_Nf$ for any $f\in L^s(X, \mu)$, see 
\cite{bou} and \cite{mt2} respectively.

The purpose of this paper is to strengthen the inequalities \eqref{eq:1} and \eqref{equ:2}
and provide strong $r$-variational estimates for the sequences $\big(\mathcal A_Nf: N\in\NN\big)$
and $\big(\mathcal H_Nf:N\in\NN\big)$. Let us recall, that for a sequence $\big(a_n : n \in A\big)$ with
$A \subseteq \ZZ$ and $r \geq 1$, the $r$-variation seminorm $V_r$ is defined by
\[
V_r\big(a_n: n \in A\big)
=\sup_{\atop{n_0<\ldots <n_J}{n_j\in A}}
\Big(\sum_{j=0}^J|a_{n_{j+1}}-a_{n_j}|^r\Big)^{1/r}.
\] 

The Calder\'{o}n transference principle allows us to reduce the matters and work on $\ZZ$
rather than on an abstract measure space $X$. Therefore the set of integers $\ZZ$ with
the counting measure and the bilateral shift operator will be our model dynamical system.
In this context, the operators \eqref{eq:a1} and \eqref{eq:s1} can be treated as convolution
operators with the kernels
\begin{equation}
	\label{eq:3}
	A_N(x) = \frac{1}{N} \sum_{p\in\PP_N}\delta_p(x)\log p,
\end{equation}
and 
\begin{equation}
	\label{eq:4}
	H_N(x) = \sum_{p\in \pm\PP_N}\delta_p(x) \frac{\log \abs{p}}{p},
\end{equation}
respectively, where $\delta_n$ stands for Dirac's delta at $n\in\ZZ$. The main result of this
paper is the following.
\begin{main_theorem}
	\label{thm:var'}
	Let $T_N$ be a convolution operator given either by \eqref{eq:3} or by \eqref{eq:4}. Then for 
	$s \in (1, \infty)$ and $r>2$ there is a constant $C_s>0$ such that
	\begin{align*}
		\big\|V_r\big(T_Nf: N\in\NN\big)\big\|_{\ell^s}
		\le
		C_s\frac{r}{r-2}\|f\|_{\ell^s},
	\end{align*}
	 for every $f\in\ell^s(\ZZ)$.
\end{main_theorem}

The advantage of studying $r$-variational seminorm is twofold. On the one hand, for each $r\geq 1$,
$r$-variations control supremum norm. More precisely, for any sequence of functions
$\big(a_n(x): n\in \NN\big)$ and any $n_0\in \NN$ we have the pointwise estimate
\[
\sup_{n\in \NN}|a_n(x)|\le |a_{n_0}(x)| + V_r\big(a_n(x): n\in \NN\big).
\]
On the other hand, $r$-variation seminorm is an invaluable tool in problems concerning pointwise 
convergence. Namely, if $ V_r\big(a_n(x): n\in \NN\big)<\infty$ then the sequence 
$\big(a_n(x): n\in \NN\big)$ converges. The second property is especially important since we obtain
a quantitative form of almost everywhere convergence of $\big(a_n(x): n\in \NN\big)$. Moreover, one
does not need to find a dense class of functions for which the pointwise convergence holds which
sometimes might be a difficult problem.

Variational estimates were the subject of many papers, see \cite{jkrw, jsw, k, zk} and the
references therein. Let us notice that Theorem~\ref{thm:var'} for the truncated singular integral
immediately implies that the singular integral along the primes
\[
\mathcal Tf(x)=\sum_{p\in\pm \PP} f(x-p) \frac{\log \abs{p}}{p}
\]
is bounded on $\ell^s(\ZZ)$ for any $s \in (1, \infty)$. This can be considered, to some extent, as an
extension of a result of Ionescu and Wainger \cite{iw} to the set of prime numbers. A multidimensional
version of Bourgain's averaging operator along polynomial mappings \cite{mt3} and truncated
version of Ionescu and Wainger singular integral \cite{iw} were considered by the first two authors
with E.~M. Stein and analogous results to Theorem \ref{thm:var'} have been obtained.
This will be the subject of a forthcoming paper.

In the last section we obtain an unweighted version  of Theorem \ref{thm:var'} for averaging
operators. Namely, let
\[
	A_N'(x)=\frac{1}{|\PP_N|}\sum_{p\in\PP_N}\delta_p(x).
\]
Then as a consequence of Theorem  \ref{thm:var'} we obtain the following.
\begin{main_theorem}
	\label{thm:var''}
	For $s \in (1, \infty)$ and $r>2$ there is a constant $C_s>0$ such that
	\[
		\big\|V_r\big(A_N' * f: N\in\NN\big)\big\|_{\ell^s}
		\le C_s\frac{r}{r-2}\|f\|_{\ell^s}
	\]
	for every $f\in\ell^s(\ZZ)$.
\end{main_theorem}
Theorem~\ref{thm:var'} extends our earlier results on the Hilbert transform \cite{mt2} and
the ergodic averages \cite{zk} and simultaneously unifies and simplifies their proofs. Its proof
proceeds in several steps. In Section \ref{sec:2} we collect some $\ell^s(\ZZ)$ estimates for
$r$-variations of some general multipliers. The general philosophy lying behind these proofs is
the transference principle which allows us to compare $\ell^s(\ZZ)$ estimates with
\textit{a priori} $L^s(\RR)$ estimates.  In Section \ref{sec:3} we use the circle method
to construct approximating multipliers which will be used to build up
$\ell^2(\ZZ)$ theory. The results in Section \ref{sec:3} are formulated in an
abstract form and are applicable to both kinds of kernels we are interested in. In Section
\ref{sec:4} we provide a systematic proof of Theorem \ref{thm:var'} based on the separate analysis
of short and long variations. The short variations are covered by the ideas from \cite{zk}.
In order to bound the long variations we use the results from Section \ref{sec:2} and 
Section \ref{sec:3}. The approach we exploit here strongly uses some specific features of the prime
numbers and this is how we preserve the dependence of the form ${r}/{(r-2)}$
on the variational exponent $r$ obtained for the corresponding continuous operators in \cite{jsw},
whereas the methods in \cite{zk} lose an additional factor of ${r}/(r-2)$.
In the last section, using Theorem \ref{thm:var'} and some transference principle for 
$r$-variations we change the weights in the averages and prove Theorem \ref{thm:var''}.

\subsection{Notation}
Throughout the paper, unless otherwise stated, $C > 0$ stands for a large positive constant
whose value may vary from occurrence to occurrence. We will say that $A\lesssim B$ ($A\gtrsim B$)
if there exists an absolute constant $C>0$ such that $A\le CB$ ($A\ge CB$). If $A\lesssim B$ and 
$A \gtrsim B$ hold simultaneously then we will shortly write that $A\simeq B$. To indicate that
the the constant depends on some $\delta > 0$ we will write $A\lesssim_{\delta} B$ 
($A\gtrsim_{\delta} B$). Let $\NN_0=\NN\cup\{0\}$. 

\subsection*{Acknowledgments}
The authors thank the Hausdorff Research Institute for Mathematics for
hospitality during the Trimester Program ``Harmonic Analysis and
Partial Differential Equations''. Especially, we are greatly indebted to Christoph Thiele for his warm 
hospitality.

The authors wish to thank Jim Wright for helpful discussions on the subject of the paper.

\section{Variational bounds}
\label{sec:2}

Here we provide some general variational estimates which allow us to
prove our main result. The proofs will be based on the Fourier
transform methods. However, we start by recalling some basic facts
from number theory. A general reference is \cite{nat}. We treat $[0, 1]$ as the circle group $\TT$, identifying $0$ and $1$.
 
For a given $q \in \NN$ let $A_q$ to be the set of all $a \in \ZZ \cap [1, q]$ such that
$(a, q) = 1$. Let $\varphi$ be Euler's totient function, which is the counting function of $A_q$.  It is well known
that for every $\epsilon > 0$ there is a constant $C_{\epsilon} > 0$ such that
\begin{equation}
	\label{eq:5}
	\varphi(q) \geq C_{\epsilon} q^{1-\epsilon}.
\end{equation}
Another arithmetic function we will need is the divisor function $d(q)$ of $q\in\NN$. We also know
that for every $\epsilon > 0$ there is a constant $C_{\epsilon} > 0$ such that
\begin{equation}
	\label{eq:6}
	d(q) \leq C_{\epsilon} q^{\epsilon}.
\end{equation}
By $\mu$ we denote M\"obius function, i.e. $\mu(1)=1$,   $\mu(q)=(-1)^n$ if $q$ is a product of $n$ distinct
prime numbers, and $\mu(q)=0$ otherwise. From the Ramanujan's identity we obtain
\[
\mu(q) = \sum_{r \in A_q}  e^{2 \pi i r a/ q} \ \ \mbox{if \  $(a, q)=1$}.
\]

Let $\mathcal{F}$ denote the Fourier transform on $\RR$ defined for any $f \in L^1(\RR)$ by
\[
\mathcal{F} f(\xi) = \int_\RR f(x) e^{2\pi i \xi x} {\rm d}x.
\]
If $f \in \ell^1(\ZZ)$ we set
\[
\hat{f}(\xi) = \sum_{n \in \ZZ} f(n) e^{2\pi i \xi n}.
\]
We fix $\eta: \RR \rightarrow \RR$ a smooth function such that $0 \leq \eta(x) \leq 1$ and
\[
\eta(x) = \begin{cases}
	1 & \text{for } \abs{x} \leq 1/4,\\
	0 & \text{for } \abs{x} \geq 1/2.
\end{cases}
\]
We may assume that $\eta$ is a convolution of two smooth functions with compact supports contained
in $[-1/2, 1/2]$. We fix $D > 32$ and set
\[
\eta_t(\xi)=\eta\big(2 \pi \cdot D^{t+2}\xi\big).
\]
Let us recall the following elementary lemma.
\begin{lemma}{\cite[Lemma 5]{mt3}}
	\label{lem:1}
	For each $t \in \NN$ and $u\in\RR$
	\begin{equation}
		\label{eq:130}
		\bigg\lVert
		\int_\TT
		e^{-2\pi i \xi j} \eta_t(\xi) {\rm d}\xi
		\bigg\rVert_{\ell^1(j)} \leq 1,
	\end{equation}
	\begin{equation}
		\label{eq:140}
		\bigg\lVert
		\int_\TT
		e^{-2\pi i \xi j}
		\big(1-e^{2\pi i \xi u}\big)\eta_t(\xi) {\rm d}\xi
		\bigg\rVert_{\ell^1(j)} \leq |u| D^{-t-2}.
	\end{equation}
\end{lemma}
For $r \geq 1$ and a sequence $\big(a_n : n \in A\big)$ with $A \subseteq \NN$, we define $r$-variation
seminorm by
\[
	V_r\big(a_n : n \in A\big) = 
	\sup_{\atop{n_0 < \cdots < n_J}{n_j \in A}}
	\Big(\sum_{j = 1}^J \abs{a_{n_j} - a_{n_{j-1}}}^r \Big)^{1/r}.
\]
Then the $r$-variation norm is given by
\[
\calV_r\big(a_n: n \in A\big)=\sup_{n \in A} \abs{a_n}+V_r\big(a_n: n\in A\big).
\]

\begin{proposition}
	\label{prop:3}
	Suppose $\big(\Phi_N: N\in\NN\big)$ is a sequence of functions on $\RR$ such that for some 
	$s \in [1, \infty)$ and $r\ge 1$ there is $B>0$ so that for any $f\in L^s(\RR) \cap L^2(\RR)$
	\begin{align}
		\label{eq:101}
		\big\|
		\mathcal V_r\big(\mathcal F^{-1}\big(\Phi_N\mathcal F f\big): N\in \NN\big)\big\|_{L^s}
		\le B\|f\|_{L^s}.
	\end{align}
	Then there is a constant $C > 0$ such that for any $t \in \NN_0$,
	$q \in [2^t, 2^{t+1})$, $m \in \{1,\ldots,q\}$ and any $f \in \ell^1(\ZZ)$
	we have
	\begin{equation}
		\label{eq:102}
		\big\|
		\mathcal V_r\big(\mathcal F^{-1}\big(\Phi_N \eta_t \hat{f}\big)(q x + m): N \in \NN\big)
		\big\|_{\ell^s(x)}
		\le C B
		\big\|
		\mathcal F^{-1}\big(\eta_t \hat{f} \big)(q x + m)
		\big\|_{\ell^s(x)}.
	\end{equation}
\end{proposition}
\begin{proof}
	First, we show that for some $C > 0$ we have
	\begin{equation}
		\label{eq:31}
		\big\|
		\mathcal V_r\big(\mathcal F^{-1}\big(\Phi_N \eta_t \hat{f}\big): N \in \NN\big)
		\big\|_{\ell^s}
		\le C B
		\big\|
		\mathcal F^{-1}\big(\eta_t \hat{f} \big)
		\big\|_{\ell^s},
	\end{equation}
	where $C$ is independent of $D$. Since $\eta_t=\eta_{t-1}\eta_t$, by H\"older's inequality we have
	\begin{multline*}
		\mathcal V_r\big(\mathcal{F}^{-1} \big(\Phi_N \eta_t \hat{f}  \big)(n): N\in\NN\big) ^s
		\leq
		\bigg(\int_{\RR}
		\mathcal V_r\big(
		\mathcal{F}^{-1} \big(\Phi_N \eta_t \hat{f}  \big)(x): N\in\NN\big)
		\big\lvert\mathcal{F}^{-1}\eta_{t-1}(n - x)\big\rvert {\rm d}x\bigg)^s\\
		\leq
		\vnorm{\mathcal{F}^{-1} \eta_{t-1}}_{L^{1}}^{s-1}
		\int_{\RR} \mathcal V_r\big(
		\mathcal{F}^{-1} \big(\Phi_N \eta_t \hat{f}  \big)(x): N\in\NN\big)^s
		\big\lvert \mathcal{F}^{-1} \eta_{t-1}(n - x) \big\rvert {\rm d}x.
	\end{multline*}	
	Notice that $\vnorm{\mathcal{F}^{-1} \eta_{t-1}}_{L^{1}}\lesssim 1$ and
	\[
	\sum_{n \in \ZZ} \big\lvert \mathcal{F}^{-1}\eta_{t-1} (n-x) \big\rvert
	\lesssim D^{-t-1}\sum_{n \in \ZZ} \frac{1}{1 + (D^{-{t-1}}(n - x))^2}
	\lesssim D^{-{t-1}}(1+D^{t+1}) \lesssim 1
	\]
	with the implied constants independent of $D$. Hence, we obtain 
	\begin{equation*}
		\big\lVert
		\mathcal V_r\big(
		\mathcal{F}^{-1} \big(\Phi_N \eta_t \hat{f}  \big): N\in\NN\big)
		\big\rVert_{\ell^s}
		\lesssim 
		\big\lVert
		\mathcal V_r\big(
		\mathcal{F}^{-1} \big(\Phi_N \eta_t \hat{f}  \big): N\in\NN\big)
		\big\rVert_{L^s}\lesssim B\big\|\mathcal{F}^{-1} \big(\eta_t\hat{f}\big)\big\|_{L^s},
	\end{equation*}
	where the last inequality is a consequence of \eqref{eq:101}. The proof of \eqref{eq:31}
	will be completed if we show that
	\[
	\big\| \mathcal{F}^{-1} \big(\eta_t\hat{f}\big) 	\big\|_{L^s}
	\lesssim \big\|\mathcal{F}^{-1} \big(\eta_t\hat{f}\big)\big\|_{\ell^s}.
	\]
	For this purpose we use \eqref{eq:140} from Lemma \ref{lem:1}. Indeed,
	\begin{multline*}
		\big\|\mathcal{F}^{-1} \big(\eta_t \hat{f}\big)\big\|_{L^s}^s
		=
		\sum_{j\in\ZZ}\int_{0}^{1} \big|\mathcal{F}^{-1} \big(\eta_t\hat{f}\big)(x+j)\big|^s 
		{\rm d}x\\
		\leq
		2^{s-1}\big\|\mathcal{F}^{-1} \big(\eta_t\hat{f}\big)\big\|_{\ell^s}^s
		+ 2^{s-1}\sum_{j\in\ZZ}\int_{0}^{1} \big|\mathcal{F}^{-1}
		\big(\eta_t\hat{f}\big)(x+j)-\mathcal{F}^{-1} \big(\eta_t\hat{f}\big)(j)\big|^s {\rm d} x\\
		=
		2^{s-1}\big\|\mathcal{F}^{-1} \big(\eta_t\hat{f}\big)\big\|_{\ell^s}^s
		+2^{s-1}\int_{0}^{1} \bigg\|\int_{-1/2}^{1/2} e^{-2\pi i \xi j}
		\big(1-e^{-2\pi i \xi x}\big)\eta_t(\xi)\hat{f}(\xi){\rm d}\xi \bigg\|_{\ell^s(j)}^s
		{\rm d}x\\
		\leq
		2^{s-1}\big\|\mathcal{F}^{-1} \big(\eta_t\hat{f}\big)\big\|_{\ell^s}^s
		+2^{s-1}\int_{0}^{1} \bigg\|\int_{-1/2}^{1/2} e^{-2\pi i \xi j}
		\big(1-e^{-2\pi i \xi x}\big)\eta_{t-1}(\xi) {\rm d} \xi\bigg\|_{\ell^1(j)}^s
		\big\|\mathcal{F}^{-1} \big(\eta_t\hat{f}\big)\big\|_{\ell^s}^s {\rm d}x\\
		\lesssim
		\big\|\mathcal{F}^{-1} \big(\eta_t\hat{f}\big)\big\|_{\ell^s}^s.
	\end{multline*}
	Next, using \eqref{eq:31} we prove \eqref{eq:102}. For each $m \in \{1, \ldots, q\}$ we define
	\[
	J_m = \big\|\mathcal V_r\big(\mathcal F^{-1}\big(\Phi_N
        \eta_t \hat{f}\big)(q x + m): N \in \NN\big)
        \big\|_{\ell^s(x)}.
	\]
	Then, by \eqref{eq:31} we obtain
	\[
	I^s =
	\sum_{m=1}^q J_m^s =
	\big\|\mathcal V_r\big(\mathcal F^{-1}\big(\Phi_N
    \eta_t \hat{f}\big): N \in \NN\big)
    \big\|_{\ell^s}^s
    \lesssim B^s
	\big\lVert
	\mathcal F^{-1}\big(\eta_t \hat{f} \big)
	\big\rVert_{\ell^s}^s.
	\]
	If $m, m' \in\{1, \ldots, q\}$ then by \eqref{eq:31} we may write
	\begin{multline*}
		\bigg\lVert \mathcal V_r\bigg( \int_{-1/2}^{1/2} 
		e^{-2\pi i \xi(x + m)} 
		\big(1 - e^{2\pi i \xi(m -m')}\big) \Phi_N(\xi) \eta_t(\xi) \hat{f}(\xi) {\rm d}\xi 
		: N \in \NN \bigg)
        \bigg\rVert_{\ell^s(x)}\\
        \leq C B\bigg\| 
		\int_{-1/2}^{1/2} e^{-2\pi i\xi x} \big(1 - e^{-2\pi i \xi(m-m')}\big) \eta_t(\xi) \hat{f}(\xi)
	 	{\rm d}\xi \bigg\|_{\ell^s(x)}.
	\end{multline*}
	Since $\eta_t = \eta_t \eta_{t-1}$, by Minkowski's inequality and \eqref{eq:140}
	the last expression may be dominated by
	\begin{multline*}
		C B\bigg\|
		\int_{-1/2}^{1/2}
		e^{-2\pi i \xi x}
		\big(1  - e^{-2\pi i \xi(m-m')}\big)
		\eta_{t-1}(\xi) \hat{f}(\xi){\rm d}\xi
		\bigg\|_{\ell^1(x)}
		\big\lVert
		\mathcal F^{-1}\big(\eta_t \hat{f} \big)
		\big\rVert_{\ell^s}\\
		\leq 
		q C B D^{-t-1}
		\big\lVert
		\mathcal F^{-1}\big(\eta_t \hat{f}\big)
		\big\rVert_{\ell^s}.
	\end{multline*}
	Since $q < 2^{t+1}$ we have $q^2 \leq D^{t+1}$, thus
	\[
	J_m \leq J_{m'} + q^{-1} C B
	\big\lVert
	\mathcal F^{-1}\big(\eta_t \hat{f} \big)
	\big\rVert_{\ell^s}.
	\]
	Raising to $s$'th power and summing up over $m' \in \{1, \ldots, q\}$ we get
	\[
	q J_m^s \leq 2^{s-1} I^s
	+ 2^{s-1} q^{1 - s} C^s B^s 
	\big\lVert
	\mathcal F^{-1}\big(\eta_t \hat{f} \big)
	\big\rVert_{\ell^s}^s
	\lesssim B^s
	\big\lVert
	\mathcal F^{-1}\big(\eta_t \hat{f} \big)
	\big\rVert_{\ell^s}^s.
	\]
	Finally, by \cite[Lemma 2]{mt2}, we have
	\[
	\big\lVert
	\mathcal{F}^{-1} \big(\eta_t \hat{f}\big)(q x + m)
	\big\rVert_{\ell^{s}(x)}
	\simeq
	q^{-1/s}
	\big\lVert
	\mathcal{F}^{-1} \big(\eta_t \hat{f}\big)
	\big\rVert_{\ell^{s}}
	\]
	which completes the proof.
\end{proof}
\begin{proposition}
	\label{prop:10}
	Suppose $\big(\Phi_N: N\in\NN\big)$ is a sequence of functions on $\RR$ such that for some
	$s \in [1, \infty)$ and $r \ge 1$ there is a constant $B>0$ so that for any 
	$f\in L^s(\RR) \cap L^2(\RR)$
	\begin{align*}
		\big\|
		\mathcal V_r\big(\mathcal F^{-1}\big(\Phi_N\mathcal F f\big): N\in \NN\big)
		\big\|_{L^s}
		\le B\|f\|_{L^s}.
	\end{align*} 
	Then for every $\epsilon>0$ there exists a constant
	$C_\epsilon > 0$ such that for each $t\in \NN_0$ and $q\in [2^t, 2^{t+1})$, and any 
	$f\in\ell^1(\ZZ)$ we have
	\begin{align}
		\label{eqm:12}
		\Big\lVert 
		\mathcal V_r \Big( \sum_{a \in A_q}
		\mathcal{F}^{-1} \big(\Phi_N(\cdot-a/q) \eta_t(\cdot-a/q)
		\hat{f}\big) : N\in\NN\Big) 
		\Big\rVert_{\ell^s}\leq
		C_\epsilon
		B 
		q^{\epsilon}
		\|f\|_{\ell^s}.
	\end{align}
\end{proposition}
\begin{proof}
	Let us recall that for any function $F$, by  the M\"obius inversion formula, we have
	\[
	\sum_{a \in A_q} F(a/q)=\sum_{b \mid q}\mu(q/b)\sum_{a=1}^b F(a/b).
	\]
	Therefore,  we obtain
	\begin{multline*}
		\Big\lVert \mathcal V_r \Big( \sum_{a \in A_q} \mathcal{F}^{-1}
		\big(\Phi_N(\cdot-a/q) \eta_t(\cdot-a/q) \hat{f}\big) : N\in\NN\Big)
		\Big\rVert_{\ell^s}\\
		\leq 
		\sum_{b \mid q} \bigg( \sum_{l=1}^q 
		\Big\lVert 
		\mathcal V_r\big(
		\mathcal{F}^{-1} \big(\Phi_N \eta_t F_b(\cdot\ ; l) \big) (q j + l) : N\in\NN
		\big) \Big\rVert_{\ell^s(j)}^s \bigg)^{1/s},
	\end{multline*}
	where for $b \mid q$ we have set
	\[
	F_b(\xi; l) = \sum_{a=1}^b \hat{f}(\xi + a/b) e^{-2\pi i l a/b}.
	\]
	Now, by Proposition \ref{prop:3} we can estimate
	\[
	\bigg(
	\sum_{l=1}^q
	\Big\lVert
	\mathcal V_r\big(
		\mathcal{F}^{-1} \big(\Phi_N \eta_t F_b(\cdot\ ; l) \big) (q j + l)
	: N\in\NN\big)
	\Big\rVert_{\ell^s(j)}^s
	\bigg)^{1/s}
	\lesssim B
	\bigg(
	\sum_{l=1}^q
	\Big\lVert
	\mathcal{F}^{-1} \big(\eta_t F_b(\cdot\ ; l) \big) (q j + l)
	\Big\rVert_{\ell^s(j)}^s
	\bigg)^{1/s}.
	\]
	Applying Minkowski's inequality we get
	\[
	\bigg(
	\sum_{l=1}^q
	\Big\lVert
	\mathcal{F}^{-1} \big(\eta_t F_b(\cdot\ ; l) \big) (q j + l)
	\Big\rVert_{\ell^s(j)}^s
	\bigg)^{1/s}
	\leq
	\Big\lVert
	\mathcal{F}^{-1} \Big(\sum_{a=1}^b \eta_t(\cdot - a/b) \Big)
	\Big\rVert_{\ell^1}
	\vnorm{f}_{\ell^s}.
	\]
	Finally, 
	\[
	\Big\lVert
	\mathcal{F}^{-1} \Big(\sum_{a=1}^b \eta_t(\cdot - a/b) \Big)
	\Big\rVert_{\ell^1}
	=
	\Big\lVert
	\mathcal{F}^{-1} \eta_t(j)
	\sum_{a=1}^b e^{-2\pi i j a/b}
	\Big\rVert_{\ell^1(j)}
	=
	b \vnorm{\mathcal{F}^{-1} \eta_t(bj)}_{\ell^1(j)}
	\]
	and since, by \cite[Lemma 2]{mt2} and \eqref{eq:130}
 	\begin{equation*}
		b \vnorm{\mathcal{F}^{-1} \eta_t(bj)}_{\ell^1(j)}
		\lesssim \vnorm{\mathcal{F}^{-1} \eta_t}_{\ell^1}
		\lesssim 1,
    \end{equation*}
	we obtain
    \[
    \bigg( \sum_{l=1}^q \Big\lVert \mathcal V_r\big(
    \mathcal{F}^{-1} \big(\Phi_N \eta_t F_b(\cdot\ ; l)
    \big) (q j + l) : N\in\NN\big) \Big\rVert_{\ell^s(j)}^s
    \bigg)^{1/s} \lesssim B\|f\|_{\ell^s}
    \]
	which, together with \eqref{eq:6} concludes the proof.
\end{proof}
Since we treat $[0, 1]$ as the circle group $\TT$, let us define  $\mathscr
R_0=\{0\}$. For $t \in \NN$ we set
\[
\mathscr{R}_t =
\big\{
	a/q \in \TT \cap\mathbb{Q}:
	2^t \leq q < 2^{t+1} \text{ and } (a, q) =1
\big\}.
\]
For a given sequence $\big(\Phi_N : N \in \NN\big)$ of functions on $\RR$ and $t \in \NN_0$
we define a sequence $\big(\nu_N^t : N \in \NN\big)$ of Fourier multipliers on $\TT$ by setting
\[
	\nu_N^t(\xi) = 
	\sum_{a/q \in \mathscr{R}_t} \frac{\mu(q)}{\varphi(q)} \Phi_N(\xi - a/q) \eta_t(\xi - a/q).
\]
\begin{theorem}
	\label{thm:11}
	Let $\big(\Phi_N: N\in\NN\big)$ be a sequence of functions such that for some
	$s \in [1, \infty)$ and $r\ge 1$ there is a constant $B>0$ such that for any
	$f \in L^s(\RR) \cap L^2(\RR)$
	\begin{align}
		\label{eq:26}
		\big\|
		\mathcal V_r\big(\mathcal F^{-1}\big(\Phi_N\mathcal F f\big): N\in \NN\big)
		\big\|_{L^s}
		\le B\|f\|_{L^s}.
	\end{align}  
	Then there exists a constant $C > 0$ and $\delta>0$ such that for
	each $t\in \NN_0$ and any $f \in \ell^1(\ZZ)$ we have
	\[
		\big\lVert 
		\mathcal V_r \big(\mathcal F^{-1}\big(\nu_{N}^{t}\hat{f}\big): N\in\NN\big)
		\big\rVert_{\ell^s}
		\leq
		C B 2^{-t\delta}\|f\|_{\ell^s}.
	\]
\end{theorem}
\begin{proof}
	Proposition \ref{prop:10} and \eqref{eq:5} immediately imply that
	\begin{align}
		\label{eq:125}
		\big\lVert \mathcal V_r 
		\big(\mathcal F^{-1}\big(\nu_{N}^{t}\hat{f}\big): N\in\NN\big)
        \big \rVert_{\ell^s}
		\leq C_\epsilon B 2^{t\epsilon}
		\|f\|_{\ell^s}
	\end{align}
	for any $\epsilon>0$. In particular we have this bound for $s=2$, however it can be refined 
	(see also \cite{bou1}). Namely, one can write
    \begin{align*}
    	\sum_{a \in A_q} \mathcal{F}^{-1} \big(\Phi_N(\cdot -a/q) \eta_t(\cdot-a/q) \hat{f} \big)
		=\sum_{a \in A_q} \mathcal{F}^{-1} \big(\Phi_N(\cdot - a/q) \eta_t(\cdot -a/q) G_q\big)
    \end{align*}
	where
    \[
	G_q(\xi) = \sum_{a \in A_q} \eta_{t-1}(\xi - a/q) \hat{f}(\xi),
	\]
	since $\eta_t = \eta_t \eta_{t-1}$, and the supports of $\eta_t(\cdot -a/q)$'s are disjoint
	when $a/q$ varies over $\mathscr{R}_t$. By \eqref{eqm:12}, we have
	\[
	\Big\lVert
	\mathcal V_r\Big( \sum_{a \in A_q} 
	\mathcal{F}^{-1} (\Phi_N(\cdot -a/q) \eta_t(\cdot -a/q) G_q) : N\in\NN\Big)
	\Big\rVert_{\ell^2}
	\lesssim
	B
	q^{\epsilon} 
	\vnorm{\mathcal{F}^{-1} G_q}_{\ell^2}.
	\]
	Thus
	\begin{multline*}
		\big\lVert \mathcal 
		V_r\big( \mathcal{F}^{-1}\big(\nu^t_N \hat{f}\big) : N\in\NN\big) \big\rVert_{\ell^2} 
		\leq 
		\sum_{q = 2^t}^{2^{t+1}-1}
		q^{-1+\epsilon}
		\Big\lVert 
		\mathcal V_r\Big(\sum_{a \in A_q} 
		\mathcal{F}^{-1} (\Phi_N(\cdot -a/q) \eta_t(\cdot -a/q) \hat{f}) : N\in\NN 
		\Big) 
		\Big\rVert_{\ell^2}\\
		\lesssim 
		B \sum_{q = 2^t}^{2^{t+1}-1} q^{-1+2\epsilon} 
		\vnorm{\mathcal{F}^{-1} G_q}_{\ell^2}\\
		\lesssim 
		B 2^{-t/2 + 2\epsilon t} 
		\Big( \sum_{a/q \in \mathscr{R}_t} \big\lVert \mathcal{F}^{-1} 
		\big(\eta_{t-1}(\cdot-a/q) \hat{f}\big)
		\big\rVert_{\ell^2}^2 \Big)^{1/2}
	\end{multline*}
    where the last estimate follows from Cauchy--Schwarz inequality and the definition of $G_q$.
	Finally, the last sum can be dominated by $\|f\|_{\ell^2}$. Hence, for appropriately chosen
	$\epsilon>0$, we obtain
	\begin{equation}
		\label{eq:23}
		\big\lVert
		\mathcal V_r\big(
		\mathcal{F}^{-1}\big(\nu_{N}^{t} \hat{f}\big)
		: N\in\NN\big)
		\big\rVert_{\ell^2} 
		\leq B 2^{-t/4} \vnorm{f}_{\ell^2}.
	\end{equation}
	To finish the proof, for $s \neq 2$ we interpolate between \eqref{eq:125} and \eqref{eq:23}.
\end{proof}

\section{Approximations of the kernels}
\label{sec:3}
To approximate the multipliers corresponding with \eqref{eq:3} and \eqref{eq:4} we adopt
the argument introduced by Bourgain in \cite{bou-p} (see also Wierdl \cite{wrl}) which is based
on Hardy and Littlewood circle method (see e.g \cite{vau}).

For any $\alpha >0$ and $N \in \NN$ {\it major arcs} are defined by
\[
\mathfrak{M}_N^\alpha 
= \bigcup_{1\le q \leq (\log N)^{\alpha}} \bigcup_{a \in A_q} \mathfrak{M}_N^\alpha(a/q)
\]
where
\[
\mathfrak{M}_N^\alpha(a/q) =
\big\{
	\xi \in \TT:
	\abs{\xi - a/q} \leq N^{-1} (\log N)^\alpha
\big\}.
\]
The set $\mathfrak{m}_N^\alpha = \TT \setminus \mathfrak{M}_N^\alpha$ is called {\it minor arc}.
\begin{theorem}
	\label{thm:2}
	Let $\big(m_N: N\in\NN\big)$ be a sequence of Fourier multipliers on $\TT$.
	Suppose there is a sequence $\big(\Psi_N : N \in \NN\big)$ of functions on 
	$\RR$ such that
	\begin{align}
		\abs{\Psi_N(\xi)} \lesssim \min\{1, N^{-1} \abs{\xi}^{-1}\}.
	\end{align}
	Assume that for each $\alpha > 32$ there is $B_\alpha > 0$ such that for all $N \in \NN$
	\begin{align}
		\label{eq:105}
		\Big\lvert
		m_N(\xi) - \frac{\mu(q)}{\varphi(q)} \Psi_N(\xi - a/q)
		\Big\rvert
		&\leq B_\alpha (\log N)^{-\alpha/8},
		\quad \text{ if }\xi \in \mathfrak{M}_N^\alpha(a/q) \cap \mathfrak{M}_N^\alpha,\\
		\label{eq:107}
		|m_N(\xi)|
		&\leq B_\alpha (\log N)^{-\alpha/8}, \quad \text{ if }
		\xi \in \mathfrak{m}_N^\alpha.
	\end{align}
	Then for each $\alpha > 32$ there is a constant $C_{\alpha}>0$ such that for all $N \in \NN$
	\[
		\sup_{\xi \in \TT} 
		\Big\lvert m_N(\xi) - \sum_{t \geq 0} \psi_N^t(\xi) \Big\rvert 
		\leq C_\alpha (\log N)^{-\alpha/8}
	\]
	where $\psi_N^t$ is a Fourier multiplier on $\TT$ defined for $t \in \NN_0$ and $N \in \NN$
	by
	\begin{equation}
		\label{eq:12}
		\psi_N^t(\xi) = \sum_{a/q \in \mathscr{R}_t} \frac{\mu(q)}{\varphi(q)}
		\Psi_N(\xi - a/q) \eta_t(\xi-a/q).
	\end{equation}
\end{theorem}
\begin{proof}
	Let us notice that for a fixed $t \in \NN$ and $\xi \in [0, 1]$ the sum \eqref{eq:12}
	consists of the single term, since $D > 32$.
	
	\vspace{2ex}
	\noindent{\bf{Major arcs estimates}:
	$\xi\in\mathfrak{M}_N^\alpha(a/q) \cap \mathfrak{M}_N^\alpha$.}
	Let $t_0$ be such that
	\begin{equation}
		\label{eq:35}
		2^{t_0} \leq q < 2^{t_0+1}.
	\end{equation}
	Next, we choose $t_1$ satisfying
	\[
	2^{t_1+1} \leq N(\log N)^{-2\alpha} < 2^{t_1+2}.
	\]
	If $t < t_1$ then for any $a'/q' \in \mathscr{R}_t$, $a'/q' \neq a/q$ we have
	\[
	\Big\lvert 	\xi - \frac{a'}{q'} \Big\rvert
	\geq
	\frac{1}{qq'} - \Big\lvert \xi - \frac{a}{q} \Big\rvert
	\geq
	2^{-t-1} (\log N)^{-\alpha} -N^{-1}(\log N)^\alpha
	\geq
	N^{-1} (\log N)^\alpha.
	\]
	Therefore, the integration by parts gives
	\[
	\abs{\Psi_N(\xi - a'/q')} \lesssim (\abs{\xi - a'/q'}N)^{-1} \lesssim (\log N)^{-\alpha}.
	\]
	Combining the last estimate with \eqref{eq:5}, we obtain that for some $\delta'>0$
	\[
	I_1=\bigg\lvert
	\sum_{t = 0}^{t_1-1} \sum_{\atop{a'/q' \in \mathscr{R}_t}{a'/q' \neq a/q}}
	\frac{\mu(q')}{\varphi(q')} \Psi_N(\xi - a'/q') \eta_t(\xi - a'/q')
	\bigg\rvert
	\lesssim
	(\log N)^{-\alpha}
	\sum_{t=0}^{t_1-1}
	2^{-\delta' t}.
	\]
	Moreover, if $\eta_{t_0}(\xi -a/q) < 1$ then $\abs{\xi - a/q} \geq 4^{-1} D^{-t_0-2}$. By
	\eqref{eq:35} we have $2^{t_0} \leq (\log N)^\alpha$. Hence by the assumptions 
	\[
	I_2=\bigg\lvert
	\frac{\mu(q)}{\varphi(q)} \Psi_N(\xi - a/q) \big(1 - \eta_{t_0}(\xi-a/q)\big)
	\bigg\rvert
	\lesssim
	D^{t_0+2}N^{-1}
	\lesssim
	(\log N)^{-\alpha}.
	\]
	In the last estimate it is important that the implied constant does not depend on $t_0$.
	Since $\Psi_N$ is bounded uniformly with respect to $N \in \NN$, by \eqref{eq:5} and the
	definition of $t_1$ we have
	\[
	I_3=\bigg\lvert
	\sum_{t = t_1}^\infty\sum_{\atop{a'/q' \in \mathscr{R}_t}{a'/q' \neq a/q}}
	\frac{\mu(q')}{\varphi(q')} \Psi_N (\xi - a'/q') \eta_t(\xi - a'/q')
	\bigg\rvert
	\lesssim
	\sum_{t = t_1}^\infty 2^{-\delta'' t}
	\lesssim \big(N^{-1}(\log N)^{2\alpha}\big)^{\delta''}\lesssim (\log N)^{-\alpha}
	\]
	for appropriately chosen $\delta''>0$. Finally, in view of \eqref{eq:105} and
	definitions of $t_0$ and $t_1$ we conclude
	\begin{align*}
		\Big\lvert m_N (\xi) - \sum_{t \geq 0} \psi_N^t(\xi) \Big\rvert
		\leq C_\alpha (\log N)^{-\alpha/8}+ I_1+I_2+I_3\lesssim (\log N)^{-\alpha/8}.
	\end{align*}
	
	\vspace*{2ex}
	\noindent{\bf Minor arcs estimates: $\xi \in \mathfrak{m}_N^\alpha$.}
	By \eqref{eq:107}, it only remains to estimate $\sum_{t \geq 0} \psi_N^t$. Let us define
	$t_1$ by setting
	\[
	2^{t_1} \leq  (\log N)^{\alpha/2} < 2^{t_1+1}.
	\]
	If $a/q \in \mathscr{R}_t$ for $t < t_1$ then $q < (\log N)^{\alpha}$ and
	\[
	\Big \lvert \xi - \frac{a}{q} \Big\rvert
	\geq 2^{-t - 1}N^{-1} (\log N)^{\alpha} \gtrsim N^{-1}(\log N)^{\alpha/2}.
	\]
    Since
	\[
	\abs{\Psi_N(\xi - a/q)} \lesssim (\abs{\xi - a/q} N)^{-1} \lesssim (\log N)^{-\alpha/2}
	\]
	the first part of the sum may be majorized by
	\[
	\Big\lvert \sum_{t = 0}^{t_1-1} \psi_N^t(\xi) \Big\vert
	\lesssim (\log N)^{-\alpha/2} \sum_{t = 0}^{\infty} 2^{-\delta' t}.
	\]
	For the second part, we proceed as for $I_3$ to get
	\[
	\Big\lvert \sum_{t = t_1}^\infty \psi_N^t(\xi)\Big\rvert \lesssim \sum_{t=t_1}^\infty
	2^{-\delta'' t} \lesssim (\log N)^{-\delta'' \alpha/2}\lesssim (\log N)^{-\alpha/8}.
	\]
	A suitable choice of $\delta', \delta''>0$ in both estimates above was possible thanks to
	\eqref{eq:5}.
\end{proof}

\begin{proposition}
	\label{prop:1}
	Let $\alpha > 32$ and $M, N\in\NN$ be such that $N (\log N)^{-\alpha/4} \le M \le N$.
	Suppose that $K$ is a differentiable function on $[M, N]$ satisfying
	\begin{align}
		\label{eq:7}
		\norm{x}\abs{K(x)}+\norm{x}^2\abs{K'(x)} \lesssim 1.
	\end{align}
	Then for each $\alpha > 32$ there is a constant $C_{\alpha}>0$ such that for all
	$\xi \in \mathfrak{M}_N^\alpha(a/q) \cap \mathfrak{M}_N^\alpha$
	\[
	\Big\lvert
	\sum_{p\in\PP_{M, N}} 
	e^{2 \pi i \xi p} K(p)\log p - \frac{\mu(q)}{\varphi(q)} \Phi_{M, N}(\xi - a/q)
	\Big\rvert
	\leq C_\alpha (\log N)^{-\alpha}
	\]
	where $\PP_{M, N}=\PP\cap(M, N]$ and 
	\[
	\Phi_{M, N}(\xi)=\int_M^Ne^{2\pi i \xi t}K(t){\rm d}t.
	\]
	The constant $C_\alpha$ depends only on $\alpha$.
\end{proposition}

\begin{proof}
	For a prime number $p$, we have  $p \mid q$ if and only if $(p\ \mathrm{mod}\ q, q) > 1$, thus
	\[
	\Big\lvert
	\sum_{\atop{1 \leq r \leq q}{(r,q) > 1}} \sum_{\atop{p \in \PP_{M, N}}{q \mid (p -r)}}
	e^{2\pi i \xi p} K(p) \log p
	\Big\rvert
	\leq 
	N^{-1} (\log N)^{\alpha/4}
	\sum_{\atop{p \in \PP}{p \mid q}} \log p
	\lesssim 
	N^{-1}
	(\log N)^{\alpha/4+1}.
	\]
	Let $\theta = \xi - a/q$ and observe that if $p \equiv r \pmod q$ then
	\[
	\xi p \equiv \theta p + ra/q  \pmod 1.
	\]
	Consequently, we have
	\begin{equation}
		\label{eq:2}
		\sum_{r \in A_q}
		\sum_{\atop{p \in \PP_{M, N}}{q \mid (p - r)}}
		e^{2\pi i \xi p} K(p) \log p
		=
		\sum_{r \in A_q} e^{2 \pi i r a/q}
		\sum_{\atop{p \in \PP_{M, N}}{q \mid (p - r)}} e^{2\pi i \theta p} K(p) \log p.
	\end{equation}
    For $x\ge 2$ let
	\[
	\psi(x; q, r) = \sum_{\atop{p \in \PP_x}{q \mid (p-r)}} \log p.
	\]
    Notice that the summation by parts applied to the inner sum on the right-hand side in
	\eqref{eq:2} gives
	\begin{multline}
		\label{eqm:6}
		\sum_{\atop{n \in (M, N]}{q \mid (n-r)}} e^{2 \pi i \theta n} K(n) \ind{\PP}(n) \log n\\
		=
		\psi(N; q, r)e^{2 \pi i \theta N} K(N)
		-\psi(M; q, r)e^{2 \pi i \theta M} K(M)
		-\int_M^{N} \psi(t; q, r)
		\frac{d}{dt} \left(e^{2\pi i \theta t}K(t) \right) {\rm d}t.
	\end{multline}
	Analogously,  we obtain
	\[
		\sum_{n \in(M, N]} e^{2\pi i \theta n} K(n)
		=Ne^{2 \pi i \theta N} K(N)
		-
		Me^{2 \pi i \theta M} K(M)
		-
		\int_{M}^{N} t \frac{d}{ {\rm d}t}
		\left( e^{2\pi i \theta t}K(t)\right) {\rm d}t.
	\]
	By Siegel--Walfisz theorem (see \cite{sieg, wal}), for every $\alpha>0$ and
	$x \geq 2$
	\begin{equation}
		\label{eq:34}
		\Big\lvert \psi(x; q, r) - \frac{x}{\varphi(q)} \Big\rvert
		\lesssim x (\log x)^{-5\alpha}
	\end{equation}
	where the implied constant depends only on $\alpha$. Therefore \eqref{eq:7} together with
	\eqref{eqm:6}--\eqref{eq:34} yield
	\begin{multline*}
		\Big\lvert
		\sum_{\atop{p \in \PP_{M, N}}{q \mid (p -r)}} e^{2\pi i\theta p} K(p) \log p
		-\frac{1}{\varphi(q)} \sum_{n \in (M, N]} e^{2\pi i \theta n} K(n)
		\Big\rvert\\
		\lesssim
		\bigg\lvert \psi(N;q,r) - \frac{N}{\varphi(q)} \bigg\rvert \abs{K(N)}
		+\bigg\lvert \psi(M;q,r) - \frac{M}{\varphi(q)} \bigg\rvert \abs{K(M)}
		+\int_M^N
		\bigg\lvert \psi(t;q,r) - \frac{t}{\varphi(q)} \bigg\rvert
		t^{-1} \big(\abs{\theta}  + t^{-1}\big) {\rm d}t\\
        \lesssim (\log N)^{-5\alpha}+N (\log N)^{-5\alpha}
		\int_{M}^{N} t^{-1} \big(
		\abs{\theta}  + t^{-1}\big) {\rm d}t
		\lesssim (\log N)^{-2\alpha}.
	\end{multline*}
	Finally, by \eqref{eq:2}, we have
	\begin{align*}
		\Big\lvert
		\sum_{r \in A_q} \sum_{\atop{p \in \PP_{M, N}}{q \mid (p -r)}}
		e^{2 \pi i \xi p} K(p) \log p 
		-\frac{\mu(q)}{\varphi(q)} \sum_{n \in(M, N]} e^{2\pi i \theta n} K(n)
	  	\Big\rvert
		\lesssim 
		q (\log N)^{-2\alpha} \leq (\log N)^{-\alpha}.
	\end{align*}
    The proof will be completed when we replace the sum by an integral. Indeed, 
	\[
	\int_{M}^{N} e^{2\pi i \theta t} K(t) {\rm d}t =
	\sum_{n\in (M, N]} \int_0^1 e^{2\pi i \theta (n + t-1)} K(n+t-1) {\rm d}t,
	\]
	thus
	\begin{multline*}
		\Big\lvert
		\sum_{n\in (M, N]}
		e^{2\pi i \theta n} K(n) -
		\int_0^1 e^{2\pi i \theta (n+t-1)} K(n+t-1) {\rm d}t
		\Big\rvert\\
		\leq
		\sum_{n\in (M, N]}
		\int_0^1 \big\lvert 1-e^{-2\pi i \theta (t-1)} \big\rvert \abs{K(n)} {\rm d}t
		+
		\sum_{n\in (M, N]} \int_0^1 \abs{K(n) - K(n+t-1)} {\rm d}t\\
		\lesssim N^{-1} (\log N)^{2\alpha},
	\end{multline*}
	which finishes the proof.
\end{proof}

\begin{proposition}
	\label{prop:5}
	Let $\alpha>32$ and $M, N\in\NN$ such that $N(\log N)^{-\alpha/4} \le M\le N$. 
	Suppose that $K$ is a differentiable function on $[M, N]$ satisfying
	\[
		\norm{x}\abs{K(x)}+\norm{x}^2\abs{K'(x)} \lesssim 1.
	\]
	Then for each $\alpha > 32$ there is $C_{\alpha}>0$, such that for all
	$\xi \in \mathfrak{m}_N^\alpha$
	\[
	\Big\lvert
	\sum_{p\in\PP_{M, N}} e^{2 \pi i \xi p} K(p)\log p
	\Big\rvert
	\leq C_\alpha (\log N)^{-\alpha/8}.
	\]
\end{proposition}
\begin{proof}
	Let 
	\[
	F_x(\xi) = \sum_{p \in \PP_x} e^{2 \pi i \xi p} \log p.
	\]
	By the summation by parts we have
	\begin{align}
		\label{eq:27}
		\Big\lvert
		\sum_{p \in \PP_{M, N}} e^{2\pi i \xi p} K (p) \log p
		\Big\rvert 
		\leq 
		|F_{N}(\xi)| \, |K(N)| +|F_{M}(\xi)| \, |K(M)|
		+ \int_{M}^{N} |F_t(\xi)| \, \sabs{ K'(t)}
		{\rm d}t.
	\end{align}
	By Dirichlet's principle one can find $(a, q) = 1$, 
	$(\log N)^\alpha \leq q \leq N(\log N)^{-\alpha}$ such that
	\[
	|\xi - a/q| \leq q^{-1} N^{-1} (\log N)^\alpha \leq q^{-2}.
	\]
	Then Vinogradov's theorem (see \cite[Theorem 1, Chapter IX]{vin} or \cite[Theorem 8.5]{nat})
	yields
	\[
	|F_t(\xi)|
	\lesssim 
	(\log N)^4 \big(N q^{-1/2} + N^{4 /5} + N^{1/2} q^{1/2}\big)
	\lesssim 
	N(\log N)^{4 - \alpha/2}
	\]
	for $t \in [M, N]$. Combining $|K'(t)| \lesssim M^{-2}$ with the last bound and \eqref{eq:27}
	we conclude
	\[
	\Big\lvert 
	\sum_{p\in\PP_{M, N}} e^{2 \pi i \xi p} K(p)\log p
    \Big\rvert 
	\lesssim (\log N)^{4 - \alpha/4}
	\lesssim (\log N)^{-\alpha/8}
	\]
	since $M\ge N (\log N)^{-\alpha/4}$ and  $\alpha > 32$.
\end{proof}

\section{Variational estimates}
\label{sec:4}
Given $\epsilon \in (0, 1)$ we set $Z_\epsilon = \{\lfloor 2^{k^\epsilon} \rfloor : k \in \NN\}$ 
(see \cite{zk}). For a sequence $\big(a_n : n \in \NN\big)$ we define {\it long $r$-variations} by
\[
	V_r^L(a_n : n \in \NN) = V_r(a_n : n \in Z_\epsilon),
\]
and the corresponding {\em short $r$-variations}
\[
	V_r^S(a_n : n \in \NN) = \Big(\sum_{k \geq 1} 
	V_r\big(a_n : n \in [N_{k-1}, N_k) \big)^r\Big)^{1/r}
\]
where $N_k = \lfloor 2^{k^\epsilon} \rfloor$.
Then
\[
	V_r(a_n : n \in \NN) \lesssim V_r^S(a_n : n \in \NN) + V_r^L(a_n : n \in \NN).
\]
\begin{proposition}
	\label{prop:4}
	Let $s \in (1, \infty)$ and assume $\big(T_N : N \in \NN\big)$ is a sequence of 
	operators satisfying
	\begin{equation}
		\label{eq:33}
		\big\lVert T_N - T_{N-1} \big\rVert_{\ell^s \rightarrow \ell^s} \lesssim N^{-1} (\log N).
	\end{equation}
	Then for any $r\ge2$ there is $\epsilon \in (0, 1)$ and $C > 0$ such that for all $f \in \ell^s(\ZZ)$
	\[
		\Big\lVert
		\Big(\sum_{k \geq 0} V_r\big(T_N f : N \in [N_k, N_{k+1})\big)^r\Big)^{1/r}
		\Big\rVert_{\ell^s}
		\leq
		C
		\lVert f \rVert_{\ell^s}
	\]
	where $N_k = \lfloor 2^{k^\epsilon} \rfloor$.
\end{proposition}
\begin{proof}
	Let $u = \min\{2, s\}$ and $0<\epsilon < \frac{u-1}{2u}$. 
	Then, by the monotonicity and Minkowski's inequality, we get
	\begin{multline*}
		\big\|V_r^S\big(T_Nf:N\in\NN\big)\big\|_{\ell^s}
		\le 
		\Big\|
		\Big(\sum_{k\ge 0}
		\Big( \sum_{N = N_k}^{N_{k+1} - 1} \abs{T_{N+1}f-T_{N}f} \Big)^{u}
		\Big)^{1/u} \Big\|_{\ell^s}\\
		\le 
		\Big(
		\sum_{k\ge 0} 
		\Big(
		\sum_{N = N_k}^{N_{k+1}-1} 
		\vnorm{T_{N+1}f-T_{N}f}_{\ell^s}
		\Big)^{u}\Big)^{1/u}.
	\end{multline*}
	Since $N_{k+1}-N_k\lesssim 2^{k^\epsilon}k^{-1+\epsilon}$, by \eqref{eq:33}, we can estimate
	\[
		\sum_{N = N_k}^{N_{k+1}-1} 
		\vnorm{T_{N+1}f-T_{N}f}_{\ell^s}
		\lesssim
		(N_{k+1}-N_k) \frac{\log N_k}{N_k}
		\lVert f \rVert_{\ell^s}
		\lesssim
		k^{-1+2\epsilon}
		\lVert f \rVert_{\ell^s}.
	\]
	Thus
	\[
		\big\lVert V_r^S\big(T_N f : N \in \NN\big) \big\rVert_{\ell^s}
		\lesssim
		\Big(\sum_{k\ge 0} k^{-u(1-2\epsilon)}\Big)^{1/u}
		\vnorm{f}_{\ell^s}
		\lesssim 
		\vnorm{f}_{\ell^s}.
	\]
\end{proof}

\begin{proposition}
	Let $s \in (1, \infty)$, $r > 2$ and $\epsilon \in (0, 1)$. Suppose $\big(T_N : N \in \NN\big)$
	is a sequence of operators on $\ell^s(\ZZ)$ such that there is a sequence 
	$(\nu_N : N \in \NN)$ of Fourier multipliers on $\TT$ such that there are $B_1, B_2 > 0$ such
	that for all $f \in \ell^1(\ZZ)$ and $k \in \NN$
	\begin{align}
		\label{eq:36}
		\big\lVert
		V_r \big(\calF^{-1}\big(\nu_N \hat{f} \big) : N \in Z_\epsilon \big)
		\big\rVert_{\ell^s}
		\leq B_1 
		\lVert f \rVert_{\ell^s},
	\end{align}
	and
	\begin{align}
		\label{eq:37}
		\big\lVert
		\big(T_{N_k} - T_{N_{k-1}}\big) f - 
		\calF^{-1}\big(\big(\nu_{N_k} - \nu_{N_{k-1}} \big)\hat{f}\big)
		\big\rVert_{\ell^s}
		\leq
		B_2
		k^{-2}
		\lVert f \rVert_{\ell^s}
	\end{align}
	where $N_k = \lfloor 2^{k^\epsilon} \rfloor$. Then there is $C > 0$ such that for all
	$f \in \ell^1(\ZZ)$
	\[
		\big\lVert
		V_r \big(T_N f : N \in Z_\epsilon \big)
		\big\rVert_{\ell^s}
		\leq
		C (B_1 + B_2)
		\lVert f \rVert_{\ell^s}.
	\]
\end{proposition}
\begin{proof}
	By triangle inequality and \eqref{eq:36}
	\begin{multline*}
		\big\lVert
		V_r\big(T_N f : N \in Z_\epsilon\big)
		\big\rVert_{\ell^s}
		\leq
		\big\lVert
		V_r\big(\calF^{-1} \big( \nu_N \hat{f} \big) : N \in Z_\epsilon \big)
		\big\rVert_{\ell^s}
		+
		\big\lVert
		V_r\big(T_N f - \calF^{-1} \big( \nu_N \hat{f} \big) : N \in Z_\epsilon \big)
		\big\rVert_{\ell^s}\\
		\leq
		B_1 \vnorm{f}_{\ell^s}
		+
		\big\lVert
		V_r\big(T_N f - \calF^{-1} \big( \nu_N \hat{f} \big) : N \in Z_\epsilon \big)
		\big\rVert_{\ell^s}.
	\end{multline*}
	To bound the second term we notice that
	\begin{multline*}
		V_r\big(T_N f - \calF^{-1}\big(\nu_N \hat{f} \big) : N \in Z_\epsilon \big) 
		\leq
		V_1\big(T_N f - \calF^{-1}\big(\nu_N \hat{f} \big) : N \in Z_\epsilon \big) \\
		\leq
		\sum_{k \geq 1}
		\big\lvert
		\big(T_{N_k} - T_{N_{k-1}}\big) f 
		- \calF^{-1}\big(\big(\nu_{N_k} - \nu_{N_{k-1}}\big) \hat{f}\big)
		\big\rvert.
	\end{multline*}
	Hence, by \eqref{eq:37} we conclude the proof.
\end{proof}

\subsection{The averages}

\begin{theorem}
	\label{thm:4}
	For each $s \in (1, \infty)$ and $r > 2$ there is $C_s > 0$ such that for all
	$f \in \ell^s(\ZZ)$
	\[
		\big\lVert
		V_r\big(A_N * f : N \in \NN \big)
		\big\rVert_{\ell^s}
		\leq
		C_s
		\frac{r}{r-2}
		\vnorm{f}_{\ell^s}.
	\]
\end{theorem}
\begin{proof}
	Let $T_N f = A_N * f$. We fix $s \in (1, \infty)$. Since $T_N$ satisfies \eqref{eq:33},
	by Proposition \ref{prop:4}, there is $\epsilon \in (0, 1)$ such that for all $r > 2$
	the short $r$-variations are bounded on $\ell^s(\ZZ)$.
	
	For the long variations, it is enough to show \eqref{eq:36} and \eqref{eq:37}. First,
	by \cite{jsw}, a sequence of multipliers $\big(\Phi_N : N \in \NN\big)$ defined by
	\[
	\Phi_N(\xi) = \int_0^1 e^{2\pi i \xi N t} {\rm d} t
	\]
	satisfies \eqref{eq:34} for all $r > 2$ with $B \lesssim r/(r-2)$. Therefore, by Theorem
	\ref{thm:11}, for $\nu_N = \sum_{t \geq 0} \nu_N^t$ we have
	\[
		\big\lVert
		V_r\big( \calF^{-1}\big(\nu_N \hat{f}\big) : N \in \NN \big)
		\big\rVert_{\ell^s}
		\leq
		C_s \frac{r}{r-2} \sum_{t \geq 0} 2^{-\delta t} \vnorm{f}_{\ell^s}
		\lesssim
		\frac{r}{r-2} \vnorm{f}_{\ell^s}.
	\]
	For the proof of \eqref{eq:37}, we notice that for any $s \in (1, \infty)$
	\begin{equation}
		\label{eq:39}
		\big\lVert
		\big(T_{N_k} - T_{N_{k-1}}\big) f 
		- \calF^{-1}\big(\big(\nu_{N_k} - \nu_{N_{k-1}}\big) \hat{f} \big)
		\big\rVert_{\ell^s}
		\lesssim
		\vnorm{f}_{\ell^s}.
	\end{equation}
	Therefore, it is enough to show that for all $\alpha > 32$ and $N \in \NN$
	\begin{equation}
		\label{eq:40}
		\big\lVert
		T_N f - \calF^{-1}\big(\nu_N \hat{f}\big)
		\big\rVert_{\ell^2}
		\lesssim
		(\log N)^{-\alpha/8}
		\vnorm{f}_{\ell^2}.
	\end{equation}
	Indeed, we can estimate
	\begin{equation}
		\label{eq:38}
		\big\lVert
		\big(T_{N_k} - T_{N_{k-1}}\big) f 
		- \calF^{-1}\big(\big(\nu_{N_k} - \nu_{N_{k-1}}\big) \hat{f} \big)
		\big\rVert_{\ell^2}
		\lesssim
		k^{-\epsilon \alpha/8}
		\vnorm{f}_{\ell^2}.
	\end{equation}
	Then, for properly chosen value of $\alpha$, interpolation between \eqref{eq:39} and
	\eqref{eq:38} gives \eqref{eq:37}.

	To prove \eqref{eq:40}, let us denote by $m_N$ the Fourier multiplier corresponding with $A_N$.
	Let
	\[
	K_N(x) = \frac{1}{N} \ind{[M, N]}(x).
	\]
	and $M = N (\log N)^{-\alpha/4}$. For
	$\xi \in \mathfrak{M}^\alpha_N(a/q) \cap \mathfrak{M}_N^\alpha$ we may write
	\begin{multline*}
		\Big\lvert
		m_N(\xi) - \frac{\mu(q)}{\varphi(q)} \Phi_N(\xi-a/q)
		\Big\rvert\\
		\leq
		\frac{M}{N} 
		\Big\lvert
		m_M(\xi)
		-
		\frac{\mu(q)}{\varphi(q)}
		\Phi_{0, M}(\xi - a/q)
		\Big\rvert
		+ 
		\Big\lvert
		\sum_{p\in\PP_{M, N}}e^{2\pi i \xi p}K(p)\log p
		-\frac{\mu(q)}{\varphi(q)} \Phi_{M, N}(\xi-a/q)
		\Big\rvert\\
		\lesssim(\log N)^{-\alpha/4}
	\end{multline*}
	where in the last estimate we have used Proposition \ref{prop:1}. Moreover, by Proposition
	\ref{prop:5}, for $\xi \in \mathfrak{m}_N^\alpha$
	\[
		\big\lvert m_N(\xi) \big\rvert \lesssim (\log N)^{-\alpha/8}.
	\]
	Hence, by Theorem \ref{thm:2} for $\Psi_N = \Phi_N$ and $\psi_N^t = \nu_N^t$ we obtain
	\eqref{eq:40}.
\end{proof}

\subsection{The truncated singular integrals}
\begin{theorem}
	For each $s \in (1, \infty)$ and $r > 2$ there is $C_s > 0$ such that for all
	$f \in \ell^1(\ZZ)$
	\[
		\big\lVert
		V_r\big(H_N * f : N \in \NN \big)
		\big\rVert_{\ell^s}
		\leq
		C_s
		\frac{r}{r-2}
		\vnorm{f}_{\ell^s}.
	\]
\end{theorem}
\begin{proof}
	We set $T_N f = H_N * f$. For a fixed $s \in (1, \infty)$, by Proposition \ref{prop:4}
	there is $\epsilon \in (0, 1)$ such that for all $r > 2$, short $r$-variations are
	$\ell^s(\ZZ)$-bounded.

	Next, we estimate the long variations. By \cite{jsw}, the sequence of multipliers
	$\big(\Phi_N : N \in \NN\big)$ defined as
	\[
		\Phi_N(\xi) = \int_{-1}^1 e^{2\pi i N t} \frac{{\rm d}t}{t}
	\]
	satisfies \eqref{eq:34} for all $r > 2$ and $B \lesssim r/(r-2)$, hence for 
	$\nu_N = \sum_{t \geq 0} \nu_N^t$ we have
	\[
		\big\lVert
		V_r\big( \calF^{-1}\big(\nu_N \hat{f}\big) : N \in \NN \big)
		\big\rVert_{\ell^s}
		\lesssim
		\frac{r}{r-2} \vnorm{f}_{\ell^s}.
	\]
	Analogously to the case of averages, to show \eqref{eq:37}, it is enough to prove
	\[
		\big\lVert
		\big(T_{N_k} - T_{N_{k-1}}\big) f 
		- \calF^{-1}\big(\big(\nu_{N_k} - \nu_{N_{k-1}}\big) \hat{f} \big)
		\big\rVert_{\ell^2}
		\lesssim
		k^{-\epsilon \alpha/8}
		\vnorm{f}_{\ell^2}
	\]
	which is a consequence of Proposition \ref{prop:1}, Proposition \ref{prop:5} and Theorem
	\ref{thm:2} applied to $N = N_k$, $M = N_{k-1}$, $m_N$ being the Fourier multiplier
	corresponding to $T_N - T_M$, $\Psi_N = \Phi_N - \Phi_M$, $\psi_N^t = \nu_N^t - \nu_M^t$ and
	\[
		K(x) = \frac{1}{x} \ind{[M, N]}(x).
	\]
\end{proof}

\section{Unweighted averages }
\label{sec:5}
The next lemma and proposition allow  us to compare averages with different weights.
In both of them the $r$-variation norm can be replaced by any other norm on sequences.
These results combined with Theorem \ref{thm:var'} will imply Theorem \ref{thm:var''}.
\begin{lemma}
	\label{lem:av-comp}
	Let $\big(\lambda_n^k : n, k \in \NN\big)$ be a sequence of non-negative real numbers such that
	\[
		\sum_{n = 1}^\infty \lambda_n^k = \Lambda < \infty
	\]
	for every $k$. Suppose that for every $N \in \NN$ the sequence
	\[
		\Big(\sum_{n = 1}^N \lambda_n^k : k \in \NN\Big)
	\]
	is decreasing. Then for any sequence $\big(a_n : n \in \NN\big)$ of complex numbers and 
	$r \geq 1$ we have
	\[
		V_r\Big(\sum_{n=1}^{\infty} \lambda_n^k a_n: N\in\NN\Big)
		\leq
		\Lambda \cdot V_r\big(a_n : n \in \NN\big).
	\]
\end{lemma}
\begin{proof}
	For each $k \in \NN$ we define a function $N_k : [0, \Lambda] \rightarrow \NN$ by
	\[
		N_k(t) = \inf\big\{N\in\NN : \sum_{i = 1}^N \lambda_i^k > t\big\}.
	\]
	Let
	\[
	I^k_n=\big\{
		t\in[0, \Lambda] : N_k(t) = n
		\big\},
	\]
	and observe that $|I^k_n|=\lambda_n^k$. Thus
	\[
		\sum_{n = 1}^\infty \lambda_n^k a_n =
		\sum_{n = 1}^\infty |I^k_n| a_n
		=
		\sum_{n = 1}^\infty \int_{I^k_n} a_n {\rm d}t
		=
		\int_0^\Lambda a_{N_k(t)} {\rm d}t.
	\]
	Hence, for any sequence of integers $0 < k_0 < k_1 < \ldots < k_J$, by Minkowski's inequality
	\begin{multline}
		\label{eq:21}
		\Big(
		\sum_{j = 1}^J
		\Big\lvert
		\sum_{n = 1}^\infty
		\big(\lambda_n^{k_j} - \lambda_n^{k_{j-1}}\big) a_n
		\Big\rvert^r
		\Big)^{1/r}
		=
		\Big(
		\sum_{j = 1}^J 
		\Big\lvert
		\int_0^\Lambda a_{N_{k_j}(t)} - a_{N_{k_{j-1}}(t)} {\rm d} t
		\Big\rvert^r
		\Big)^{1/r} \\
		\leq
		\int_0^\Lambda
		\Big(
		\sum_{j = 1}^J 
		\Big\lvert
		a_{N_{k_j}(t)} - a_{N_{k_{j-1}}(t)} 
		\Big\rvert^r
		\Big)^{1/r}
		{\rm d} t
	\end{multline}
	Since for a fixed $t \in [0, \Lambda]$
	\[
		t < \sum_{n = 1}^{N_{k+1}(t)} \lambda_n^{k+1} 
		\leq \sum_{n = 1}^{N_{k+1}(t)} \lambda_n^k,
	\]
	we have $N_k(t) \leq N_{k+1}(t)$. Therefore, the right-hand side in \eqref{eq:21} can be
	bounded by 
	\[
		\Lambda \cdot V_r\big(a_n : n \in \NN\big).
	\]
\end{proof}

\begin{proposition}
	\label{cor:av-comp}
	Let $\big(w_n : n \in \NN\big)$ and $\big(w_n' : n \in \NN\big)$ be non-negative
	sequences satisfying one of the following conditions:
	\begin{enumerate}
		\item 
			\label{dec}
			the sequence $\big(w_n'/w_n : n \in \NN\big)$ decreases monotonically, or
		\item 
			\label{inc}
			the sequence $\big(w_n'/w_n : n \in \NN\big)$ increases monotonically and
			\[
				C=\sup_{N \in \NN} \frac{W_N w_N'}{W_N' w_N} < \infty,
			\]
		where $W_N = \sum_{n = 1}^N w_n$ and $W_N' = \sum_{n = 1}^N w_n'$.
	\end{enumerate}
	Then there is $C' > 0$ such that for each sequence $\big(a_n : n \in \NN \big)$ of complex
	numbers and $r \geq 1$
	\[
		V_r\Big(\sum_{n = 1}^N w_n a_n : N \in \NN\Big)
		\leq
		C' \cdot 
		V_r\Big(\sum_{n = 1}^N w_n' a_n : N \in \NN\Big).
	\]
\end{proposition}
\begin{proof}
	Let
	\[
		A_N = \sum_{n = 1}^N w_n a_n, \qquad A_N' = \sum_{n = 1}^N w_n' a_n.
	\]
	By partial summation we have
	\[
		A_k' = \sum_{n=1}^{\infty} \lambda_n^k A_n,
	\]
	where
	\[
		\lambda_n^k =
		\begin{cases}
			\frac{W_n}{W_k'} \Big(\frac{w_n'}{w_n} - \frac{w_{n+1}'}{w_{n+1}} \Big) 
			& \text{if } 1\le  n<k,\\
			\frac{W_k}{W_k'}\frac{w_k'}{w_k} & \text{if } n=k,\\
			0 & \text{otherwise.}
		\end{cases}
	\]
	Let us observe that for each $k \in \NN$ and $N \in \NN$
	\begin{equation}
		\label{sum1}
		\sum_{n=1}^N \lambda_n^k =
		\begin{cases}
			\frac{1}{W_k'}\Big(W_N' - W_N \frac{w_{N+1}'}{w_{N+1}}\Big)
			& \text{if } 1 \leq N < k,\\
			1 & \text{otherwise.}
		\end{cases}
	\end{equation}
	Moreover, for $1 \leq N < k$
	\[
		\frac{1}{W_k'}\Big(W_N' - W_N \frac{w_{N+1}'}{w_{N+1}}\Big) \leq 1,
	\]
	thus the sequence
	\[
		\Big(\sum_{n = 1}^N \lambda^k_n : k \in \NN\Big)
	\]
	is decreasing. Since in the case (i), $\lambda_n^k$ are non-negative, we may directly apply
	Lemma \ref{lem:av-comp}.

	For the proof in the case of (ii), we define
	\[
		\tilde{\lambda}_n^k = 
		\begin{cases}
			-\lambda_n^k   & \text{if } 1 \leq n<k,\\
			2C-\lambda_k^k & \text{if } n=k,\\
			0              & \text{otherwise.}
		\end{cases}
	\]
	Therefore, $\tilde{\lambda}_n^k$ are non-negative and
	\[
		A_k' = 2 C A_k - \sum_{n = 1}^\infty \tilde{\lambda}_n^k A_n.
	\]
	Let us observe that 
	\[
		\sum_{n = 1}^N \tilde{\lambda}_n^k
		=
		\begin{cases}
			\frac{1}{W_k'}\big(W_{N+1} \frac{w_{N+1}'}{w_{N+1}} - W_{N+1}'\big) 
			& \text{if } 1 \leq N < k,\\
			2C - 1
			& \text{otherwise.}
		\end{cases}
	\]
	Since $C > 1$ and for $1 \leq N < k$
	\[
		\frac{1}{W_k'}\Big(W_{N+1} \frac{w_{N+1}'}{w_{N+1}} - W_{N+1}'\Big) \leq C,
	\]
	the sequence
	\[
		\Big(\sum_{n = 1}^N \tilde{\lambda}_n^k : k \in \NN\Big)
	\]
	is decreasing. Hence, by Lemma \ref{lem:av-comp}
	\[
		V_r\big(A_N' : N \in \NN\big) 
		\leq
		2 C \cdot V_r\big(A_N : N \in \NN\big)
		+
		V_r\Big(\sum_{n = 1}^k \tilde{\lambda}_n^k A_n : k \in \NN\Big)
		\leq
		(4 C - 1) \cdot V_r\big(A_N : N \in \NN\big).
	\]
\end{proof}
\begin{proof}[Proof of Theorem \ref{thm:var''}]
	Let $T_N = A_N' * f$ where
	\[
		A_N'(x) = \frac{1}{\abs{\PP_N}} \sum_{p \in \PP_N} \delta_p(x).
	\]
	Fix $s \in (1, \infty)$. By Proposition \ref{prop:4}, there is $\epsilon \in (0, 1)$ such that
	for each $r > 2$ the corresponding short $r$-variations are $\ell^s(\ZZ)$-bounded.
	
	Let $w'_n \equiv 1$, $w_n = \log n$, and $a_n = \ind{\PP}(n)$. By Proposition \ref{cor:av-comp}
	we get
	\begin{equation}
		\label{eq:28}
		V_r\big(A_N' * f(x) : N \in Z_\epsilon \big)
		\lesssim
		V_r\Big(\frac{N}{W_N} A_N * f(x) : N \in Z_\epsilon \Big).
	\end{equation}
	Since for any $\beta > 0$
	\[
		W_N = N \big(1 + \calO\big((\log N)^{-\beta}\big)\big),
	\]
	by \eqref{eq:28} and Theorem \ref{thm:4} we get
	\begin{multline*}
		\big\lVert
		V_r\big(A_N' f : N \in Z_\epsilon \big)
		\big\rVert_{\ell^s}
		\leq
		\big\lVert
		V_r\big(A_N f : N \in Z_\epsilon \big)
		\big\rVert_{\ell^s}
		+
		\big\lVert
		V_r\big(\big(1 - N/W_N \big)A_N f : N \in Z_\epsilon \big)
		\big\rVert_{\ell^s}\\
		\lesssim
		\frac{r}{r-2} \vnorm{f}_{\ell^s}
		+
		\sum_{k \geq 1} k^{-\beta \epsilon}
		\lVert A_{N_k} f \rVert_{\ell^s}
	\end{multline*}
	which finishes the proof.
\end{proof}

\begin{bibliography}{discrete}
	\bibliographystyle{plain} 
\end{bibliography}

\end{document}